\documentclass[conference]{IEEEtran}
\ifCLASSINFOpdf
\else
\fi

\usepackage{mathrsfs}
\usepackage{amssymb}
\usepackage{amsmath,bm}

\usepackage{mathrsfs}
\usepackage{graphicx}
\usepackage{eurosym}
\usepackage{amssymb}
\usepackage{amsmath}
\usepackage{amsfonts}
\usepackage{epstopdf}
\usepackage{epsf,subfigure}
\usepackage{psfrag}
\usepackage{graphics}
\usepackage{color} 
\usepackage{cite}
\usepackage{array,multirow,pbox}
\usepackage{enumitem}
\usepackage{caption}

\newcommand{\rem}[1]{}

\newtheorem{theorem}{Theorem}

\newtheorem{remark}{Remark}
\newtheorem{problem}{Problem}

\newenvironment{proof}[1][Proof]{\begin{trivlist}
\item[\hskip \labelsep {\bfseries #1}]}{\end{trivlist}}

\newcommand{\qed}{\nobreak \ifvmode \relax \else
      \ifdim\lastskip<1.5em \hskip-\lastskip
      \hskip1.5em plus0em minus0.5em \fi \nobreak
      \vrule height0.75em width0.5em depth0.25em\fi}

\graphicspath{{Figures/}}

\newcommand{\mysubeq}[2]{
\begin{subequations}\label{#1}
\begin{align}
#2
\end{align}\end{subequations}}

\begin{document}
\title{\vspace{0.25in}{Distributed Barrier Certificates for Safe Operation of Inverter-Based Microgrids}}

\author{\IEEEauthorblockN{Soumya Kundu, Sai Pushpak Nandanoori, Karan Kalsi}
\IEEEauthorblockA{Optimization and Control Group\\
Pacific Northwest National Laboratory, USA\\
Email:\{soumya.kundu,\,saipushpak.n,\,karanjit.kalsi\}@pnnl.gov}
\vspace{-0.35in}
\and
\IEEEauthorblockN{Sijia Geng, Ian A. Hiskens}
\IEEEauthorblockA{Electrical Engineering and Computer Science\\
University of Michigan - Ann Arbor, USA\\
Email:\{sgeng,\,hiskens\}@umich.edu}
\vspace{-0.35in}}


%\author{\IEEEauthorblockN{Soumya Kundu\IEEEauthorrefmark{1},
%Sijia Geng\IEEEauthorrefmark{2},
%Sai Pushpak Nandanoori\IEEEauthorrefmark{1}, 
%Ian Hiskens\IEEEauthorrefmark{2} and
%Karan Kalsi\IEEEauthorrefmark{1}}
%\IEEEauthorblockA{\IEEEauthorrefmark{1}Electrical Engineering and Computer Science\\
%University of Michigan, Ann Arbor, MI 48109 USA\\ Email:}
%\IEEEauthorblockA{\IEEEauthorrefmark{2}Optimization and Control Group\\
%Pacific Northwest National Laboratory, Richland, WA 99352 USA\\
%Email:}}


\maketitle

\begin{abstract}
Inverter-interfaced microgrids differ from the traditional power systems due to their lack of inertia. Vanishing timescale separation between voltage and frequency dynamics makes it critical that faster-timescale stabilizing control laws also guarantee by-construction the satisfaction of voltage limits during transients. In this article, we apply a barrier functions method to compute distributed active and reactive power setpoint control laws that certify satisfaction of voltage limits during transients. Using sum-of-squares optimization tools, we propose an algorithmic construction of these control laws. Numerical simulations are provided to illustrate the proposed method.

\end{abstract}


% \fi
\IEEEpeerreviewmaketitle

\section{Introduction}

Safety critical system refers to a system for which the violation of safety constraints will lead to serious economic loss or personal casualty. Power system falls into such category considering the loss resulting from large-scale blackout and critical loads such as hospital and process plant. Modern power system has been evolving towards distributed operation. With the increasing integration of distributed energy resources (DERs), especially renewable resources, challenges have arisen in safely operating power systems as well as guaranteeing stability. Microgrid is a promising direction to tackle the intermittency and uncertainty characteristics that are intrinsic in renewable resources such as wind and solar. An islanded microgrid is a standalone small scale power system that groups a variety of DERs, especially renewables, together with energy storages and loads to provide better control and operation, higher efficiency and reliability \cite{pogaku2007modeling}. It is a viable solution for power supply to rural area. Microgrid also provides a new perspective to increase the penetration level of renewables in modern power system. Unlike traditional power system that has large inertia from conventional synchronous generators, DERs in a microgrid are connected to the network through power electronic interface. Considering the stochasticity in renewables and the negligible physical inertia, control of voltage and frequency for microgrids is challenging \cite{mahmoud2014modeling}. While stability analysis and control of inverter-based microgrids have received a lot of attention in the literature \cite{simpson2013synchronization,Schiffer:2014,barklund2008energy,vasquez2013modeling,hiskens2008control}, safety of the microgrids has largely been ignored. For microgrids, the safe region can be defined for voltage magnitudes at every node in the network. The transient voltage in a microgrid can fluctuate by a large amount, causing serious power quality and safety issues, even causing damages to the electrical equipment \cite{kusko2007power,Xu:2018}. Flexible power injections at the droop-controlled inverter nodes can be utilized to stabilize the phase angle, frequency and voltage magnitude, as well as ensuring the voltage magnitudes at all nodes within the safe region. 

Control Lyapunov function (CLFs) have long been used to synthesize stabilizing controllers for nonlinear systems \cite{sontag1989universal}. On the other hand, barrier functions are used to certify safety by guranteeing the forward invariance of a set via Lyapunov-like conditions. Although barrier function originated in the field of optimization as a penalty function to replace constraints, it prospers in the field of control design too. For example, \cite{Prajna:2007} considers the safety verification problem in both worst-case and stochastic settings by constructing barrier certificates. The ideas of the barrier functions and the CLF were combined to construct the control barrier functions (CBFs) \cite{Wieland:2007} which have since been used in designing safety controllers. Reference \cite{xu2018correctness} applied CBF method to automotive systems to achieve lane keeping and adaptive cruise control simultaneously with safety constraints. Application of CBF method to establish set invariance with the existence of disturbance and uncertainty is considered in \cite{gurriet2018towards}. 
Simultaneous satisfaction of safety and performance objectives via design is not a trivial task. The stabilization objective expressed by a CLF and the safety constraints established by a CBF can be potentially in conflict. In \cite{Ames:2017}, the authors proposed a quadratic program formulation that unifies CLF and CBF to synthesize a controller that enforces the safety constraints but relaxes the stability (performance) requirement when these two objectives are in conflict. \cite{Wang:2018} proposed an iterative algorithm using sum-of-squares (SOS) technique to search for the most permissive barrier function that gives maximum volume for the certified region, therefore maximize the estimate for safe stabilization region. All trajectories that start within the safe stabilization region can be made to converge to the (equilibrium of interest at the) origin as well as constrained in the safe region. 

The main contribution of this paper is in applying barrier functions based method to certify safety of an inverter-based microgrid considering transient voltage limits. We propose a distributed safety certification method and present computational algorithms to compute safety-ensuring decentralized and distributed control policies. To treat the control design problem in a decentralized perspective, the microgrid is firstly decomposed into several subsystems. Barrier functions are generated for each subsystem by firstly ignoring the interactions from neighboring subsystems. The interaction terms are considered as disturbances with upper limits in the control design phase, resulting in robust local state feedback control strategies. The rest of the article is organized as follows: Section\,\ref{S:back} presents the necessary background; Section\,\ref{S:model} explains the microgrid model; the main computational and algorithmic developments are described in Section\,\ref{S:algo}, with numerical results presented in Section\,\ref{S:resul}. We conclude the article in Section VI. Throughout the text, we will use $\left|\,\cdot\,\right|$ to denote both the Euclidean norm of a vector and the absolute value of a scalar; and use $\mathbb{R}\left[x\right]$ to denote the ring of all polynomials in $x\in\mathbb{R}^n$.

\vspace{-0.01in}
%===========================================
\section{Background}\label{S:back}
%===========================================
\vspace{-0.01in}

%-------------------------------------------------
\subsection{Stability Certificates: Lyapunov Functions}
%-------------------------------------------------

Consider a nonlinear dynamical system of the form 
\begin{align}\label{E:f}
\dot{x}(t) &= f(x(t))~ \forall t \geq 0\,,~x\in\mathbb{R}^n\,,
\end{align}
with an equilibrium at the origin $(f(0)= 0)$, where $f:\mathbb{R}^n \rightarrow \mathbb{R}^n$ is locally Lipschitz. For brevity, we would drop the argument $t$ from the state variables, whenever obvious. The equilibrium point at the origin of \eqref{E:f} is Lyapunov stable if, for every $\varepsilon\!>\!0$ there is a $\delta\!>\!0$ such that $\left\vert x(t)\right\vert\!<\!\varepsilon~\forall t\!\geq\! 0$ whenever $\left\vert x(0)\right\vert\!<\!\delta\,.$ Moreover, it is asymptotically stable in a domain $\mathcal{X}\!\subseteq\!\mathbb{R}^n,\,0\!\in\!\mathcal{X},$ if it is Lyapunov stable and  $\lim_{t\rightarrow\infty}\left\vert x(t)\right\vert \!=\!0\,$ for every $x(0)\!\in\!\mathcal{X}$\,.
%and it is exponentially stable if there exists $b,\,c > 0$ such that $\left\vert x(t)\right\vert <ce^{-bt}\left\vert  x(0)\right\vert \,\,\forall t\geq 0\,$, for every $\left\vert x(0)\right\vert \in \mathcal{X}$.

%\begin{theorem}\label{T:Lyap}
%\cite{Lyapunov:1892,Khalil:1996} Let there be a continuously differentiable radially unbounded positive definite function $V\!:\!\mathcal{X}\!\rightarrow\!\mathbb{R}_{\geq 0}$ such that $\nabla_x V^T\!f(x)$ is negative definite in $\mathcal{X}$. Then the origin of \eqref{E:f} is asymptotically stable and $V(x)$ is a Lyapunov function.
%\end{theorem}
\begin{theorem}\label{T:Lyap}
\cite{Lyapunov:1892,Khalil:1996} If there is a continuously differentiable radially unbounded positive definite function $V\!:\!\mathcal{X}\!\rightarrow\!\mathbb{R}_{\geq 0}$ such that $\nabla_x V^T\!f(x)$ is negative definite in $\mathcal{X}$, then the origin of \eqref{E:f} is asymptotically stable and $V(x)$ is a Lyapunov function.
\end{theorem}

Here $\nabla_x$ denotes the partial differentiation with respect to $x$\,. Using an appropriately scaled Lyapunov function $V(x)$\,, the region-of-attraction (ROA) of the origin of \eqref{E:f} can be estimated by $\left\lbrace x\in\mathcal{X}\left| {V}(x)\leq 1\right.\right\rbrace$ \cite{Genesio:1985,Anghel:2013}.

%-------------------------------------------------
\subsection{Safety Certificates: Barrier Functions}
%-------------------------------------------------

In contrast to asymptotic stability which concerns with the convergence of the state variables to the stable equilibrium, the notion of `safety' comes from engineering design specifications. From the design perspective, the system trajectories are not supposed to cross into the certain regions in the state-space marked as `unsafe'. Let us assume that the `unsafe' region of operation for the system \eqref{E:f} is given by the domain 
\begin{align}
\mathcal{X}_u:=\lbrace x\in\mathbb{R}^n\left\vert \,w_i(x)> 0\,,~i=1,2,\dots,l\right.\rbrace
\end{align}
where $w_i:\mathbb{R}^n\mapsto\mathbb{R}$ are a set of $l$ ($\geq 1$) polynomials. These are usually engineering constraints that ensure that the system is always operated (controlled) to avoid going into `unsafe' modes of operation. Safety of such systems can be verified through the existence (or, construction) of continuously differentiable barrier functions $B:\mathbb{R}^n\mapsto\mathbb{R}$ of the form \cite{Prajna:2007,Wieland:2007,Ames:2017,Wang:2018}:
\begin{subequations}\label{E:B}
\begin{align}
B(x)&\geq 0\quad \forall x\in\mathbb{R}^n\backslash\mathcal{X}_u\\
B(x)&<0\quad \forall x\in\mathcal{X}_u\\
 (\nabla_x B)^T\!f(x)+\alpha\left(B(x)\right) &\geq 0\quad \forall x\in\mathbb{R}^n
\end{align}
\end{subequations}
where $\alpha(\cdot)$ is an extended class-$\mathcal{K}$ function\footnote{A continuous function $\alpha:(-a,b)\mapsto(-\infty,\infty)$\,, for some $a,b>0\,,$ is \textit{extendend class-$\mathcal{K}$} if it is strictly increasing and $\alpha(0)=0$ \cite{Khalil:1996}.}. The third condition ensures that at the level-set $B=0$ the value of the barrier function is increasing along the system trajectories. Safety is guaranteed for all trajectories starting inside the domain $\lbrace x\left|\,B(x)\geq 0\right.\rbrace$ which is \textit{invariant} under the dynamics \eqref{E:f}.

%-------------------------------------------------
\subsection{Sum-of-Squares Optimization}
%-------------------------------------------------
Relatively recent studies have explored how SOS-based methods can be utilized to find Lyapunov functions by restricting the search space to SOS polynomials \cite{Wloszek:2003,Parrilo:2000,Tan:2006,Anghel:2013}. 
Let us denote by $\mathbb{R}\left[x\right]$ the ring of all polynomials in $x\in\mathbb{R}^n$. 
%\begin{definition}
A multivariate polynomial $p\in\mathbb{R}\left[x\right],~x\in\mathbb{R}^n$, is an SOS if there exist some polynomial functions $h_i(x), i = 1\ldots s$ such that 
$p(x) = \sum_{i=1}^s h_i^2(x)$.
We denote the ring of all SOS polynomials in $x$ by $\Sigma[x]$.
%\end{definition} 
Whether or not a given polynomial is an SOS is a semi-definite problem which can be solved with SOSTOOLS, a MATLAB$^\text{\textregistered}$ toolbox \cite{sostools13}, along with a semi-definite programming solver such as SeDuMi \cite{Sturm:1999}. 
%The SOS technique can be used to search for polynomial LFs by translating the conditions in Theorem\,\ref{T:Lyap} to equivalent SOS conditions \cite{sostools13,Wloszek:2003,Wloszek:2005,Antonis:2005,Antonis:2005a, Chesi:2010a }. 
An important result from algebraic geometry, called Putinar's Positivstellensatz theorem \cite{Putinar:1993,Lasserre:2009}, helps in translating conditions such as in \eqref{E:B} into SOS feasibility problems. 
\begin{theorem}\label{T:Putinar}
Let $\mathcal{K}\!\!=\! \left\lbrace x\in\mathbb{R}^n\left\vert\, k_1(x) \geq 0\,, \dots , k_m(x)\geq 0\!\right.\right\rbrace$ be a compact set, where $k_j$ are polynomials. Define $k_0=1\,.$ Suppose there exists a $\mu\!\in\! \left\lbrace {\sum}_{j=0}^m\sigma_jk_j \left\vert\, \sigma_j \!\in\! \Sigma[x]\,\forall j \right. \right\rbrace$ such that $\left\lbrace \left. x\in\mathbb{R}^n \right\vert\, \mu(x)\geq 0 \right\rbrace$ is compact. Then,  
\begin{align*}
p(x)\!>\!0~\forall x\!\in\!\mathcal{K}
\!\implies\! p \!\in\! \left\lbrace {\sum}_{j=0}^m\sigma_jk_j \left\vert\, \sigma_j \!\in\! \Sigma[x]\,\forall j \right. \right\rbrace\!.
\end{align*}
\end{theorem}

\begin{remark}\label{R:Putinar}
Using Theorem\,\ref{T:Putinar}, one can translate the problem of checking that $p\!>\!0$ on $\mathcal{K}$ into an SOS feasibility problem where we seek the SOS polynomials $\sigma_0\,,\,\sigma_j\,\forall j$ such that $p\!-\!\sum_j\sigma_j k_j$ is {SOS.
%, which can be efficiently solved using semi-definite programming. Note that Theorem\,\ref{T:Putinar} can easily handle equality constraints, since each 
Note that any} equality constraint $k_i(x)\!=\!0$ can be expressed as two inequalities $k_i(x)\!\geq 0$ and $k_i(x)\!\leq\! 0$. In many cases, especially for the $k_i\,\forall i$ used throughout this work, a $\mu$ satisfying the conditions in Theorem\,\ref{T:Putinar} is guaranteed to exist (see \cite{Lasserre:2009}), and need not be searched for.
\end{remark}

%===========================================
\section{Microgrid Model}\label{S:model}
%===========================================

We consider the following model of droop-controlled inverter dynamics \cite{Coelho:2002,Schiffer:2014}:
\mysubeq{E:droop}{
\dot{\theta}_i & = \omega_i\,,\\
\tau_i\dot{\omega}_i & = -\omega_i + \lambda_i^p \left(P_i^{\text{set}}-P_i\right)\\
\tau_i\dot{v}_i & = v_i^0-v_i + \lambda_i^q \left(Q_i^{\text{set}}-Q_i\right)
}
where $\lambda_i^p>0$ and $\lambda_i^q>0$ are the droop-coefficients associated with the active power vs. frequency and the reactive power vs. voltage droop curves, respectively; $\tau_i$ is the time-constant of a low-pass filter used for the active and reactive power measurements; $\theta_i\,,\,\omega_i$ and $v_i$ are, respectively, the phase angle, speed and voltage magnitude; $v^0_i$ is the desired (nomial) voltage magnitude; $P_i^{\text{set}}$ and $Q^{\text{set}}_i$ are the active power and reactive power set-points, respectively. Finally, $P_i$ and $Q_i$ are, respectively, the active and reactive power injected into the network which relate to the neighboring bus voltage phase angle and magnitudes as:
\mysubeq{E:PQ}{
P_i &= v_i{\sum}_{k\in\mathcal{N}_i} v_k\left(G_{i,k}\cos\theta_{i,k} + B_{i,k}\sin\theta_{i,k}\right)\\
Q_i &= v_i{\sum}_{k\in\mathcal{N}_i} v_k\left(G_{i,k}\sin\theta_{i,k} - B_{i,k}\cos\theta_{i,k}\right)
}
where $\theta_{i,k}=\theta_i-\theta_k$\,, and $\mathcal{N}_i$ is the set of neighbor nodes. $G_{i,k}$ and $B_{i,k}$ are respectively the transfer conductance and susceptance values of the line connecting the nodes $i$ and $k$\,. 

At the equilibrium (steady-state) operation:
\begin{align*}
\forall i:\quad P_i=P_i^{\text{set}},\,Q_i=Q_i^{\text{set}},\,\omega_i=0,\,v_i=v_i^0\,.
\end{align*} 
As the conditions on the network change (such as changes in load or generation), inverters have the capability to change the control set-points of the active and reactive power output to adjust to the new operating conditions. This is modeled as:
\begin{align}
P_i^{\text{set}} = P^0_i+u^p_i\,,~Q_i^{\text{set}} = Q^0_i+u^q_i\,,
\end{align}
where $P^0_i$ and $Q^0_i$ are the set-points for the unperturbed (or nominal) operating condition; and $u^p_i$ and $u^q_i$ are some feedback control inputs. 

Due to the low-inertia of the microgrids, large voltage and frequency fluctuations are quite common during transients \cite{Xu:2018}. While designing stabilizing control policies, it is therefore important to keep track of the transient voltage and frequency magnitudes to ensure that those are within the `safety' limits determined via engineering design. In this work, we will restrict ourselves to the consideration of transient voltage limits which are usually higher than the steady-state operational limits \cite{kusko2007power}. Fluctuations of transient voltage limits beyond the tolerable (`safe') region cause power quality issues, including the risk of damaging the electrical equipment. In this paper, we will define the `safe' operational region as:
\begin{align*}
\underline{v_i}\leq v_i(t)\leq \overline{v_i}\,.
\end{align*}
Typical values for the limits during transients operation could be $\underline{v_i}=0.6$\,p.u. and  $\overline{v_i}=1.2$\,p.u. Other forms of safety constraints, such as frequency limits and power-flow limits will be considered in the future work.

%===========================================
\section{Distributed Safety Certificates}\label{S:algo}
%===========================================

The dynamical model of the interconnected microgrid with $m$ droop-controlled inverters is expressed compactly as:
\mysubeq{E:network}{
\dot{x}_i&=f_i(x_i)+g_i(x_i)u_i + {\sum}_{j\in\mathcal{N}_i}h_{ij}(x_i,x_j)\,,\\
\mathcal{X}_{u,i}&:=\lbrace x_i\,|\,w_j(x_i)\geq 0\,,~j=1,2,\dots,l_i\rbrace
}
where each $i\in\lbrace 1,2,\dots,m\rbrace$ identifies an inverter. $x_i\in\mathbb{R}^{n_i}$ is the $n_i$-dimensional state vector associated with the $i$-th inverter, while $u_i$ are some control inputs. We assume that the origin is an equilibrium point of interest of the networked system, and that the control input vanishes at the equilibrium point (i.e. $u_i=0\,\forall i$ at the origin). We also assume that $h_{ij}(x_i,0)=0$ for all $x_i$\,. Moreover $f_i,\,g_i$ and $h_{ij}$ are locally Lipschitz functions. The problem we are interested in is:
\begin{problem}
Identify continuous functions $B_i(x_i)$\,, feedback control policies $u_i$ and non-negative scalars $c_i$ such that 
\mysubeq{E:conditions}{
\forall i:~&B_i(0)>c_i\label{E:inclusion}\\
&B_i(x_i)<0\quad \forall x_i\in\mathcal{X}_{u,i}\label{E:safety}\\
 &\dot{B}_i \geq 0\quad \forall x_i\in\partial\mathcal{D}_i[c_i],\,\forall x_j\in\mathcal{D}_j[c_j]~\forall j\in\mathcal{N}_i\label{E:derivative}\\
 &\dot{B}_i=\nabla_{x_i}B_i^T(f_i(x_i)+g_i(x_i)u_i + {\sum}_{j\in\mathcal{N}_i}h_{ij}(x_i,x_j))\,.\notag
}
where we define $\mathcal{D}_i[c_i]:=\lbrace x_i\,|\,B_i(x_i)\geq c_i\rbrace~\forall i$ and $\partial\mathcal{D}_i[c_i]:=\lbrace x_i\,|\,B_i(x_i)=c_i\rbrace$ as the boundary set of the domain $\mathcal{D}_i[c_i]$\,.
\end{problem}

\begin{theorem}
If there exist continuous functions $B_i(x_i)$\,, feedback control policies $u_i$ and non-negative scalars $c_i$ satisfying \eqref{E:conditions}, then the safety of the interconnected system \eqref{E:network} is guaranteed for all $t\geq 0$ whenever $B_i(x_i(0))\geq c_i\,\forall i$\,, i.e.
\begin{align*}
x_i(0)\in\mathcal{D}_i[c_i]~\forall i\implies x_i(t)\in\mathbb{R}^{n_i}\backslash\mathcal{X}_{u,i}~\forall i\,\forall t\geq 0\,.
\end{align*}
Moreover there is a neighborhood $\mathcal{X}_i$ around origin (i.e. $0\in\mathcal{X}_i\,\forall i$) such that $\mathcal{X}_i\subseteq\mathcal{D}_i[c_i]$\,.
\end{theorem}
\begin{proof}
Note that because of the condition \eqref{E:derivative}, $B_i$ is non-decreasing on the boundary of the domain $\mathcal{D}_i[c_i]$ whenever $x_j\in\mathcal{D}_j[c_j]$ for every neighbor $j$\,. Extending this argument to all the subsystems, we conclude that 
\begin{align*}
\mathcal{D}_1[c_1]\times\mathcal{D}_2[c_2]\times\dots\times\mathcal{D}_m[c_m]
\end{align*}
is an invariant domain. Since $B_i(x_i)<0$ for every $x_i\in\mathcal{X}_{u,i}$, we conclude the safety of the system is guaranteed for all $t\geq 0$ whenever $x_i(0)\in\mathcal{D}_i[c_i]~\forall i$\,. Finally, since $B_i(0)>c_i$ and $B_i$ is a continuous function there exists a neighborhood $\mathcal{X}_i$ around origin such that for all $x_i\in\mathcal{X}_i$\,, $B_i(x_i)\geq c_i$\,.
\hfill\qed\end{proof}

Computation of such barrier functions is not trivial. Recent works have explored the use of sum-of-square optimization methods to compute the barrier certificates for polynomials networks \cite{Prajna:2007,Wang:2018}\,. Note that the power-flows as described in \eqref{E:PQ} are non-polynomial. Using the polynomial recasting technique proposed in \cite{Anghel:2013}\,, the power systems dynamics can be expressed in a higher-dimensional space as a polynomial differential-algebraic system. In this work, however, we resort to Taylor series expansion (up to third order) to approximate the dynamics into a polynomial form. 

In the rest of this section, we describe a three-step procedure to obtain the distributed barrier certificates. In the first two steps, we consider the isolated and autonomous sub-system model of the form (which we assume to be locally asymptotically stable around the origin):
\begin{align*}
\text{(isolated)}\quad \dot{x}_i=f_i(x_i)\,,
\end{align*}
and compute the Lyapunov function which is then used to compute a barrier function for the isolated sub-system using the method similar to \cite{Wang:2018}.

%-------------------------------------------------
\subsection{Computation of Lyapunov Functions}
%-------------------------------------------------
SOS-based expanding interior algorithm \cite{Wloszek:2003,Anghel:2013} has been used to construct Lyapunov functions and region-of-attraction in an iterative search process. In this work, we use a variant of the process - which does not require the bisection search process and hence speeds up the computation at each iteration stage. The algorithmic steps used to implement the modified expanding interior algorithm can be summarized as follows (for notational convenience, we have dropped the sub-script $i$ from the subsystem variables to explain the algorithm):
\begin{enumerate}
\item Step 0: Compute a Lyapunov function $V^0$ such that $V^0\leq 1$ is an estimate of the ROA. The following two steps are then repeated until convergence, such that (hopefully) the final estimate of the ROA is much larger than initial one. Define a positive definite and radially unbounded function $p(x)$ (e.g. $p=\varepsilon|x|^2$ for some small $\varepsilon>0$).
\item Step k-1: Starting from a Lyapunov function $\hat{V}$ with ROA estimated by $\hat{V}\leq 1$\,, compute the largest level-set $\beta^k$ of a positive definite function $p(x)$ contained within $\hat{V}\leq 1$\,. This is done by solving the following SOS problem:
\mysubeq{E:LF_step1}{
\max_{s_2^k,s_3^k,s_4^k}\quad\beta^k\\
\hat{s_1}(p-\beta^k) - s_2^k(\hat{V}-1)\in\Sigma[x]\\
-s_3^k(1-\hat{V})-s_4^k\dot{\hat{V}}-\varepsilon_2|x|^2\in\Sigma[x]
}
where $s$ are SOS polynomials. The $\hat{\cdot}$ implies it is borrowed from the previous step, while $k$ denotes the variables being currently computed. At the first instance of the problem \eqref{E:LF_step1}, we initialize $\hat{s_1}=1$\,.
\item Step k-2: In this sub-step at the k-th iteration, a new Lyapunov function $V^k$ is found such that the level-set $V^k= 1$ is an estimate of the ROA, while trying to expand the estimated ROA by maximizing $\delta$ such that $p\leq\hat{\beta}$ is contained within the level-set $V^k\leq 1-\delta^k$, i.e.
\mysubeq{}{
\max_{V^k,s_1^k}\quad\delta^k\\
V^k-\varepsilon_1|x|^2\in\Sigma[x]\\
s_1^k(p-\hat{\beta}) - \hat{s_2}({V}^k-1+\delta^k)\in\Sigma[x]\\
-\hat{s_3}(1-{V}^k)-\hat{s_4}\dot{{V}}^k-\varepsilon_2|x|^2\in\Sigma[x]
}
\end{enumerate}
The algorithm stops when $\delta^k$ is sufficiently small. Set $V=V^k$ as the Lyapunov function with $V\leq 1$ providing the largest estimate of the ROA\,.

%-------------------------------------------------
\subsection{Computation of Barrier Functions}
%-------------------------------------------------
For the barrier functions computation we adopt a similar approach as in the Algorithm\,2 in \cite{Wang:2018}, except that we use the algorithm to compute only the barrier functions, while the original algorithm was used to also search for a `safe' and stabilizing control policy. For completeness we present the algorithm here (once more, for notational convenience, we have dropped the sub-script $i$ from the subsystem variables to explain the algorithm):
\begin{enumerate}
\item Step 0: As the first step of the iterative process, we compute the maximum level-set of $V$ contained completely inside the \textit{safe region}. This is done by solving the following SOS problem:
%\begin{align*}
%\max_{s_0^k} \quad z\\
%V-z-{\sum}_{i=1}^ls_{0,i}^kw_i\in\Sigma[x]\,.
%\end{align*}
\begin{align*}
\max_{s_0^k}~z\,,~\,\text{s.t.}~V-z-\sum_{i=1}^ls_{0,i}^kw_i\in\Sigma[x]\,.
\end{align*}

Set $B^0=z^{\max}-V$\,, where $z^{\max}$ is the solution of the above problem, i.e. the maximal level-set of $V$ wholly contained inside the safe region. Note that $B^0$ is a barrier function by construction. Choose a small scalar $\gamma\!>\!0$\,.
\item Step k-1: Using the barrier function $\hat{B}$ computed in the previous step, find the largest $\varepsilon\!>\!0$ such that $\dot{\hat{B}}\geq-\gamma \hat{B}+\varepsilon$ whenever $\hat{B}\geq 0$\,, i.e. solve the SOS problem
%\begin{align*}
%\max_{s_1^k}\quad\varepsilon^k\\
%\dot{\hat{B}}+\gamma\,\hat{B}-\varepsilon^k-s_1^k\,\hat{B}\in\Sigma
%\end{align*}
\begin{align*}
\max_{s_1^k}~\varepsilon^k,~\,\text{s.t.}~
\dot{\hat{B}}+\gamma\,\hat{B}-\varepsilon^k-s_1^k\,\hat{B}\in\Sigma
\end{align*}
\item Step k-2: In this sub-step we search for a new barrier function of the form $B^k(x)=z(x)^TQ^kz(x)$ where $z(x)$ is a vector of monomials in $x$\,, and $Q^k$ is a symmetric matrix, such that $B^k(0)>0$\,. The barrier function satisfies $B^k(x)<0$ on the unsafe set $\lbrace x\,|\,w_i(x)> 0\,,\,i=1,2,\dots,l\rbrace$\,, along with the constraint on its time-derivative. The following problem is solved:
\begin{align*}
\max_{s_{2,i}^k}\quad\text{trace}\,(Q^k)\\
\dot{B}^k+\gamma\,B^k-\eta-\hat{s_1}\,B^k\in\Sigma[x]\\
-B^k -{\sum}_{i=1}^ls_{2,i}^kw_i\in\Sigma[x]
\end{align*}
where $\eta$ is a small positive number chosen to avoid the trivial zero solution. The objective function is a proxy for maximizing the volume of the safety region \cite{Wang:2018}.
\end{enumerate}
The algorithm stops when $\text{trace} (Q^k)$ converges within some tolerance. Set $B=B^k$ as the barrier function with $B\geq 0$ providing the largest estimate of the certified safety region for the isolated subsystems.

%-------------------------------------------------
\subsection{Safety Certifying Control Policies}
%-------------------------------------------------
%In the previous two sections, we described how (isolated) subsystem Lyapunov functions ($V_i(x_i)$) can be computed, using which the barrier functions ($B_i(x_i)$) for the isolated subsystems are computed which satisfy the condition \eqref{E:safety}\,. 
In this subsection, we describe an SOS problem to compute control policies $u_i$ such that \eqref{E:derivative} is satisfied for some $c_i\in[0,B_i(0))$\,. Without any loss of generality, we will assume that the barrier functions satisfy $B_i(0)=1$ (always achievable through scaling), such that we are interested in $c_i\in[0,1)$\,. Moreover, we will assume, for simplicity, uniform $c_i=c~\forall i$\,, while the more generic case can be easily extended. Then we are seeking the existence of control laws $u_i$ such that for some chosen $c\in[0,1)$ 
\mysubeq{E:control}{
\forall i:~&\forall x_i\in\partial\mathcal{D}_i[c]\,,\,\forall x_j\in\mathcal{D}_j[c],\,j\in\mathcal{N}_i\\
&\nabla_{x_i}B_i^T(f_i+g_iu_i+{\sum}_{j\in\mathcal{N}_i}h_{ij})\geq 0~
}
In this paper, we will focus on state-feedback control policies. Two alternatives will be considered: 1) a decentralized state-feedback policy of the form $u_i(x_i)$\,, and 2) a distributed state-feedback policy of the form $u_i=u_{ii}(x_i)+{\sum}_{j\in\mathcal{N}_i}u_{ij}(x_j)$\,.

The following problem concerns the design of an optimal decentralized state-feedback control policy $u_i(x_i)$:
\mysubeq{E:control_sos}{
&\qquad\min_{u_i(x_i)}\quad U_i\\
&\text{s.t.}\quad \text{\eqref{E:control} and }~\|u_i(x_i)\|_\infty\leq U_i\,~\forall x_i\in\mathcal{D}_i[c]\,.
}
Note that the controller is only used on or near the boundary of the domain $\mathcal{D}_i[c]$ since it is only needed to guarantee that the trajectories never cross the boundary. This can be solved using an equivalent SOS problem, noting that the constraint $\|u_i(x_i)\|_\infty\leq U_i$ translates to polynomial constraints. Similar problem can be formulated for the distributed controller design, with the constraint $\|u_i(x_i)\|_\infty\leq U_i$ needed to be satisfied on $x_j\in\mathcal{D}_j[c]\,\forall j\in\mathcal{N}_i$ as well as $x_i\in\mathcal{D}_i[c]\,$.

\begin{remark}
Note that the constraint \eqref{E:control} is satisfied whenever (sufficient condition) we choose a $u_i$ such that 
\begin{align*}
\nabla B_i^Tg_iu_i\geq \mu:=\max_{x_i\in\partial{D}_i[c],x_j\in\mathcal{D}_j[c]}\left|{\sum}_{j\in\mathcal{N}_i}\nabla B_i^Th_{ij}\right|
\end{align*}
If for every $x_i\in\partial{D}_i[c]$ there always exists a $k$ such that $\left|\left[\nabla B_i^T\right]_k\left[u_i\right]_k\right|>0$ then we can always find a $u_i$ satisfying the above condition. 
\end{remark}

%===========================================
\section{Numerical Results}\label{S:resul}
%===========================================
%
%\begin{figure}[thpb]
%\centering
%\includegraphics[scale=0.25]{SPIDERS.eps}
%\caption{Microgrid network adopted from \cite{ersal2011impact}.}
%\label{F:spiders}
%\end{figure}
%
%For illustration purpose, we consider the microgrid network (Fig.\,\ref{F:spiders}) described in \cite{ersal2011impact}\,. Disconnecting the utility, we replace the substation (bus 0) by a droop-controlled inverter, with three other inverters placed on buses 1, 4 and 5\,. The inverter dynamics were modeled in the form of \eqref{E:droop}. Bus 0 was considered as the reference bus for the phase angle.
For illustration purpose, we consider the 6-bus (bus 0 to bus 5) microgrid network described in \cite{ersal2011impact}\,. Disconnecting the utility, we replace the substation (bus 0) by a droop-controlled inverter, with three other inverters placed on buses 1, 4 and 5\,. The inverter dynamics were modeled in the form of \eqref{E:droop}. Bus 0 was considered as the reference bus for the phase angle. 
The droop coefficients $\lambda^p_i$ and $\lambda^q_i$ were chosen to be $2.43\,$rad/s/p.u. and $0.2$\,p.u./p.u., while the filter time-constant $\tau_i$ was set to 0.5\,s \cite{Schiffer:2014}. Nominal values of voltage and frequency, as well as the active and reactive power set-points were obtained by solving the steady-state power-flow equations \eqref{E:PQ}, which were then used to shift the equilibrium point to the origin. The loads were modeled as constant power loads, and Kron reduced network with only the inverter nodes were used for analysis.
The unsafe set was defined in terms of the shifted (around the 1\,p.u.) nodal voltage magnitudes as follows:
\begin{align*}
\text{(unsafe)}\quad v_i< -0.4\,\text{p.u.}~\text{ or }~\,v_i> 0.2\,\text{p.u.}
\end{align*}
\begin{figure}[thpb]
\centering
\vspace{-0.2in}
\includegraphics[scale=0.45]{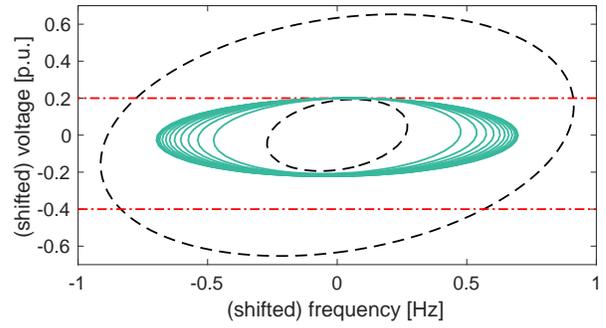}
\caption{Illustration of the iterative search for a barrier certified region of an isolated inverter subsystem. Red dashed line mark the boundary of the unsafe region. The outer black dashed lines mark the estimated ROA, while the inner black dashed line marks the largest Lyapunov functions level-set contained within the safe region. Green lines mark the iterative (growing) estimates of the certified safe region using barrier function.}
\label{F:inv4_barrier}
\end{figure}
In Fig.\,\ref{F:inv4_barrier}, we illustrate how the iterative search algorithm presented in Section\,\ref{S:algo} obtains an expanded certified region of safety (marked by the boundary of the outermost green ellipse) starting from the initial estimate given by a level-set of the Lyapunov function (marked by the smaller black dashed ellipse boundary). The plot shows the projections of the ROA and the barrier certified regions on the frequency-voltage space obtained by setting the phase angle differences to 0. The black dash marked boundary of larger ellipse is the estimate of the ROA, while the red dashed lines denote the unsafe region boundary (in voltage magnitudes). Note that the certified invariant region of safe stability is much smaller than the estimated ROA of the isolated inverter.
%
%
%%
%\begin{figure}[thpb]
%\centering
%\includegraphics[scale=0.45]{control_dec.eps}
%\caption{Computation of the minimum control effort needed ($U_i$) to certify safety of the network via subsystem barrier function, for varying value of $c\in[0,1)$\,, using decentralized control policy $u_i(x_i)$.}
%\label{F:control_dec}
%\end{figure}
%%
%
%%
%\begin{figure}[thpb]
%\centering
%\includegraphics[scale=0.33]{control_dist.eps}
%\caption{Computation of the minimum control effort needed ($U_i$) to certify safety of the network via subsystem barrier function, for varying value of $c\in[0,1)$\,, using distributed control policy that uses all the neighbor states into the feedback as well.}
%\label{F:control_dist}
%\end{figure}
%%
%
%%
%\begin{figure}[thpb]
%\centering
%\includegraphics[scale=0.45]{control_d.eps}
%\caption{Computation of the minimum control effort needed ($U_i$) to certify safety of the network via subsystem barrier function, for varying value of $c\in[0,1)$\,, using distributed control policy that uses the neighbor voltage magnitudes into the feedback.}
%\label{F:control_d}
%\end{figure}
%%
%
\begin{figure*}[thpb]
\centering
\hspace{-0.3in}
\subfigure[]{
\includegraphics[scale=0.32]{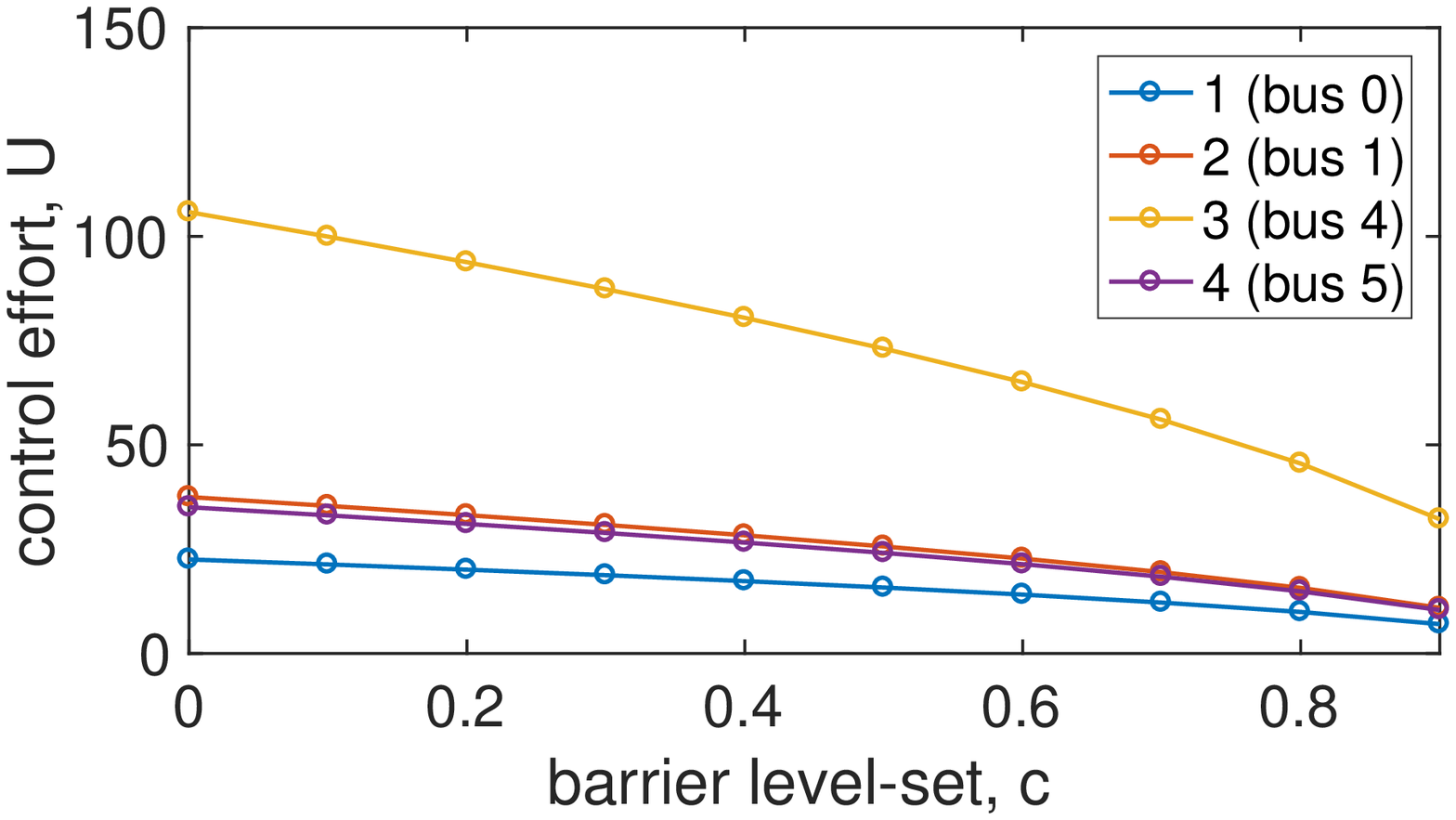}\label{F:control_dec}
}
\hspace{-0.41in}
\subfigure[]{
\includegraphics[scale=0.32]{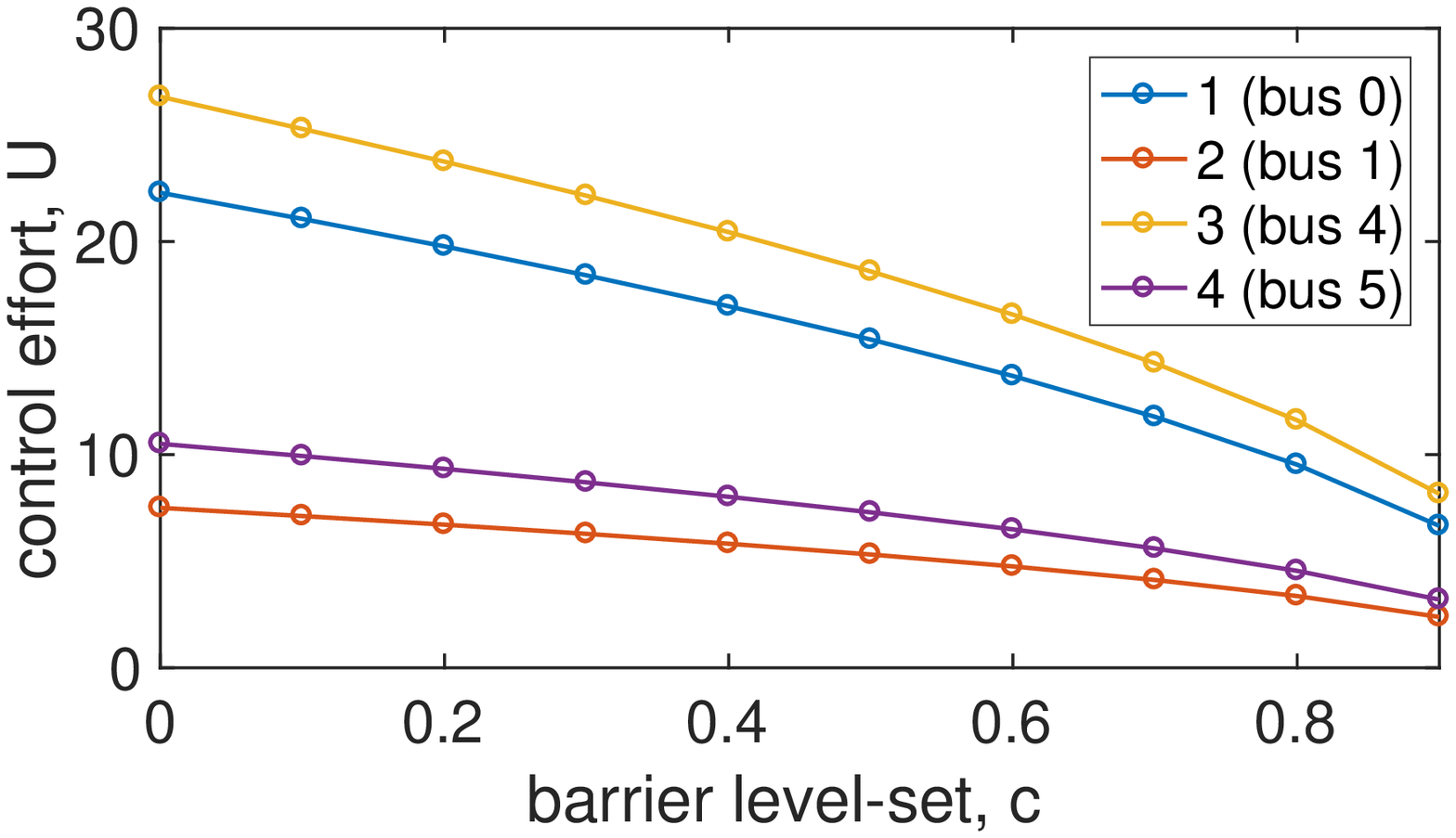}\label{F:control_dist}
}
\hspace{-0.41in}
\subfigure[]{
\includegraphics[scale=0.32]{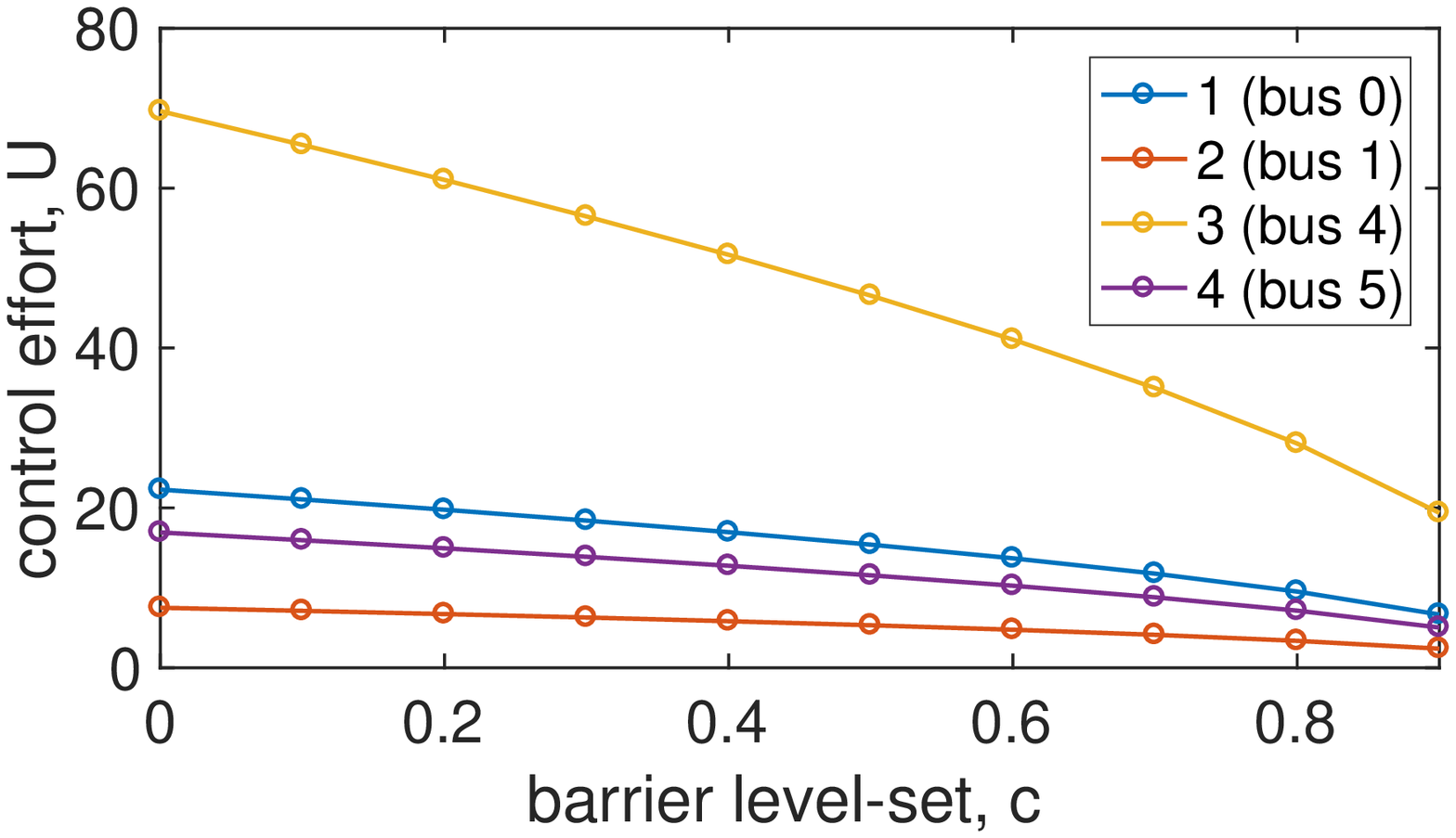}\label{F:control_d}
}\hspace{-0.41in}
\caption[Optional caption for list of figures]{Computation of the minimum control effort needed ($U_i$) to certify safety of the network via subsystem barrier functions, for varying values of $c\in[0,1)$\,, using - (a) a decentralized control policy, $u_i(x_i)$\,, that used only the subsystem's states; (b) a distributed control policy that uses all the neighbor states into the feedback, in addition to the subsystem's states; and (c) a distributed control policy that uses the neighbor voltage magnitudes into the feedback, in addition to the subsystem's states.}
\label{F:c1}
\end{figure*}
Next we investigate the control efforts needed to guarantee safety of the network over some domain defined using the subsystem barrier function level-sets. Fig.\,\ref{F:control_dec} shows the results of the optimal decentralized control design problem \eqref{E:control_sos}, for a range of different values of the barrier level-set $c\in[0,1)$ such that $B_i\geq c~\forall i$ gives a distributed certificate of safety. As expected, the control effort increases (monotonically, in this case) as the value of $c$ decreases, or the certified region of safety increases. Figs.\,\ref{F:control_dist}-\ref{F:control_d} show the results when local neighboring subsystem state measurements are used in the control design in addition to the subsystem's states, in what we call as the `distributed control' design. Clearly, distributed control policies require lower minimum control efforts as compared to the decentralized control policy. This observation aligns with the conclusion in \cite{cavraro2016value} regarding the value of communication in distribution network voltage regulation problem. In particular, two different choices of distributed controllers are explored - one in which all of the neighboring subsystems' states are used in the feedback (Fig.\,\ref{F:control_dist}) and another in which the only neighboring subsystem states used as feedback are the voltage magnitudes (Fig.\,\ref{F:control_d}). In this example, additional measurements from the neighboring subsystems help decrease the minimum control effort needed.

\vspace{-0.02in}
%===========================================
\section{Conclusion}
%===========================================
\vspace{-0.02in}
In this paper we consider the problem of safety in inverter-based microgrids. Using the barrier functions based methods, we propose a distributed safety certification of the mirogrid network. Sum-of-squares based algorithm was used to present a computational approach to obtain these safety certificates in a distributed manner. Moreover, using a microgrid example, we show how decentralized vs. distributed control policies could pose different requirements on the control effort. Future work will explore the extension of such methods to larger power systems networks, under the presence of uncertainties.


\vspace{-0.03in}
\section*{Acknowledgment}
\vspace{-0.03in}
This work was carried out at PNNL (contract DE-AC05-76RL01830) under the support from the U.S. Department of Energy as part of their Grid Modernization Initiative. 

\vspace{-0.02in}


\bibliographystyle{IEEEtran}
\bibliography{references,RefKundu,RefBarrier,RefMGStability}
%\begin{thebibliography}{1}
%\end{thebibliography}

\end{document}